\documentclass[10pt,11pt]{article}
\usepackage{amssymb}
\usepackage{amsmath}
\usepackage{amsfonts}

\newtheorem{theorem}{Theorem}
\newtheorem{definition}{Definition}

\newtheorem{remark}{Remark}
\newtheorem{example}{Example}
\newtheorem{lemma}{Lemma}
\newenvironment{condition0}{\noindent\textbf{[B.0]} }
{}
\newenvironment{condition1}{\noindent\textbf{[A.0]} }
{}
\newenvironment{condition2}{\noindent\textbf{[A.1]} }
{}
\newenvironment{condition3}{\noindent\textbf{[A*.2]} }
{}
\newenvironment{condition4}{\noindent\textbf{[A*.3]} }
{}

\newenvironment{proof}{\noindent\textbf{Proof.}}{}
\evensidemargin=\oddsidemargin
\newdimen\dummy

\dummy=\oddsidemargin 
\input{tcilatex}

\begin{document}

\title{On some probabilistic properties of periodic $GARCH$ processes}
\author{\textsc{Abdelouahab Bibi}$^{\text{*}}$ \and \textsc{Abdelhakim
Aknouche}$^{\text{**}}$ \\
%EndAName
$^{\text{*}}$D\'{e}partement de Math\'{e}matiques, Universit\'{e} Mentouri
de Constantine, Algeria\\
E-mail: abd.bibi@gmail.com\\
$^{\text{**}}$Facult\'{e} de Math\'{e}matiques, Universit\'{e} U.S.T.H.B.,
Algiers, Algeria\\
E-mail: aknouche\_ab@yahoo.com}
\date{}
\maketitle

\begin{abstract}
\noindent This paper examines some probabilistic properties of the class of
periodic $GARCH$ processes ($PGARCH$) which feature periodicity in
conditional heteroskedasticity. In these models, the parameters are allowed
to switch between different regimes, so that their structure shares many
properties with periodic $ARMA$ process ($PARMA$). We examine the strict and
second order periodic stationarities, the existence of higher-order moments,
the covariance structure, the geometric ergodicity and $\beta -$mixing of
the $PGARCH(p,q)$ process under general and tractable assumptions. Some
examples are proposed to illustrate the various concepts.


\noindent {\normalsize \textbf{AMS (2000) Subject Classification.} Primary
62M10; Secondary 62M05.}
\end{abstract}

\section{Introduction}

Consider a periodic $GARCH(p_{1},...,p_{s},q_{1},...,q_{s})$ process $\left(
x_{t}\right) _{t\in \mathbb{Z}}$ with period $s>0$ and orders $p=\left(
p_{1},...,p_{s}\right) $ and $q=\left( q_{1},...,q_{s}\right) $, defined on
some probability space $\left( \Omega ,\mathcal{A},P\right) $ with the
non-linear periodic difference equation: 
\begin{equation}
\forall t\in \mathbb{Z}\text{: }\left\{ 
\begin{array}{l}
x_{st+v}=\varepsilon _{st+v}\sqrt{h_{st+v}} \\ 
h_{st+v}=\alpha _{0}(v)+\sum\limits_{i=1}^{p_{v}}\alpha
_{i}(v)x_{st+v-i}^{2}+\sum\limits_{j=1}^{q_{v}}\beta _{j}(v)h_{st+v-j}%
\end{array}%
\text{ }\right.  \tag{1.1}  \label{Eq1_1}
\end{equation}%
where $\left( \varepsilon _{t}\right) _{t\in \mathbb{Z}}$ is a sequence of
independent and identically distributed (i.i.d.) random variables defined on
the same probability space $\left( \Omega ,\mathcal{A},P\right) $ such that $%
E\left\{ \varepsilon _{t}\right\} =E\left\{ \varepsilon _{t}^{3}\right\} =0,$
and $E\left\{ \varepsilon _{t}^{2}\right\} =1$ (these conditions are
obviously satisfied if $\left( \varepsilon _{t}\right) _{t\in \mathbb{Z}}$
is Gaussian).

In the difference equation $\left( \ref{Eq1_1}\right) $ $x_{st+v}$ refers to 
$x_{t}$ during the $v-th$ `season', $1\leq v\leq s$ of period $t$, $\alpha
_{0}(v),\alpha _{1}(v),...,$ $\alpha _{p_{v}}(v)$ and $\beta
_{1}(v),...,\beta _{q_{v}}(v)$ are the model coefficients at season $v$ such
that for all $v=1,...,s$, $\alpha _{0}(v)>0,$ $\alpha _{i}(v)\geq 0$, $%
i=1,...,p_{v}$, and $\beta _{j}(v)\geq 0$, $j=1,...,q_{v}$. Moreover, we
assume that $\varepsilon _{k}$ is independent of $x_{t}$ for $k>t$. We use
the periodic notations \ $\left( x_{st+v}\right) ,\left( \varepsilon
_{st+v}\right) ,\left( h_{st+v}\right) ,\left( \alpha _{i}(v),0\leq i\leq
p_{v}\right) $, and $\left( \beta _{i}(v),1\leq i\leq q_{v}\right) $ to
emphasize the periodicity in the model. There is no loss of generality in
taking $p_{v}$ and $q_{v}$ to be constant in $v$. If $p_{v}$ or $q_{v}$
change with $v$, one can set $p=\max\limits_{1\leq v\leq s}p_{v}$, $%
q=\max\limits_{1\leq v\leq s}q_{v}$ and take $\alpha _{k}(v)=0$ for $%
p_{v}<k\leq p$ and $\beta _{k}(v)=0$ for $q_{v}<k\leq q$, so in the sequel,
we shall consider the periodic $GARCH$ with constant orders $p$ and $q$.
Since Bollerslev and Ghysels $(1996)$, this type of non-linear models has
become an appealing tool for investigating both volatility and distinct
seasonal patterns, and has been applied in various disciplines such as
finance and monetary economics (see e.g. Bollerslev and Ghysels,$1996$ and
Franses and Paap, $2000$).

When we consider a periodic model as a data generating process, it is
important to find conditions ensuring the (periodic) stationarity,
ergodicity and the existence of higher moments for further statistical
analysis. Various probabilistic properties of standard $GARCH$ models have
been studied extensively by many authors (see e.g., Chen and An, $1998$,
Bougerol and Picard, $1992a$, $1992b$ and Carrasco and Chen, $2002$ and the
references therein). In the present paper, we focus on studying the
fundamental probabilistic properties of the $PGARCH$ process $\left(
x_{t}\right) _{t\in \mathbb{Z}}$ generated by $\left( \ref{Eq1_1}\right) $
so, in Section $2$, we present a vectorial representation from which we
derive some sufficient conditions for the strict stationarity. In Sections $%
3 $ and $4$, necessary and sufficient conditions for the second order
stationarity and the existence of higher order moments are given. Section $5$
is devoted to covariance structure. In Section $6$ we provide conditions
under which strictly stationary solutions are exponential $\beta $-mixing
with finite higher order moments. We conclude in Section $7$.

Some notations are used throughout the paper: $I_{(k)}$ denotes the identity
matrix of order $k$ and $O_{(k\times l)}$ denotes the matrix of order $%
k\times l$ whose elements are zeroes, for simplicity we set $%
O_{(k)}:=O_{(k\times k)}$, $\rho \left( A\right) $ refers to the spectral
radius of a square matrix $A$, i.e., the maximum eigenvalue of a matrix $A$
in absolute value, $Vec\left( A\right) $ is the usual column-stacking vector
of the matrix $A$, $\left\Vert .\right\Vert $ refers to the standard
(Euclidean) norm in $\mathbb{R}^{n}$ or the uniform induced norm in the
space $\mathcal{M}(n)$ of $n\times n$ matrices, $\otimes $ denotes the
Kronecker product of matrices, and $A^{\otimes m}=A\otimes A\otimes
...\otimes A$ (\emph{m}$-$times), for any integer $m\geq 1$. For any $p\geq
1 $, $\mathbb{L}^{p}=\mathbb{L}^{p}\left( \Omega ,\mathcal{A},P\right) $
denotes the Hilbert space of random variables $X$ defined on the probability
space $\left( \Omega ,\mathcal{A},P\right) $ such that $\left\Vert
X\right\Vert _{p}=\left\{ E\left\vert X\right\vert ^{p}\right\}
^{1/p}<+\infty $. We also use the following property of matrix operation, $%
Vec\left( ABC^{\prime }\right) =(C\otimes A)Vec(B)$, where $^{\prime }$ is
the matrix transpose.

\section{The Markovian representation and strict stationarity}

Let $\left( x_{t}\right) _{t\in \mathbb{Z}}$ be a process conforming to the
model $\left( \ref{Eq1_1}\right) $. Setting $y_{st+v}=x_{st+v}^{2}$ and\ $%
\eta _{st+v}=\varepsilon _{st+v}^{2}$, we obtain \ from $\left( \ref{Eq1_1}%
\right) $ the following representation 
\begin{equation}
y_{st+v}=\sum\limits_{i=1}^{p}\alpha _{i}(v)\eta
_{st+v}y_{st+v-i}+\sum\limits_{j=1}^{q}\beta _{j}(v)\eta
_{st+v}h_{st+v-j}+\alpha _{0}(v)\eta _{st+v}.  \tag{2.1}  \label{Eq2_1}
\end{equation}

\noindent Equation $\left( \ref{Eq2_1}\right) $ is intractable when we want
to examine the probabilistic structure of this representation. Instead, we
will work with the corresponding state-space representation. Let $d=p+q$ and
define\ $\underline{\eta }_{t}=(\eta _{st+1},...,\eta _{st+s})^{\prime }$, $%
\underline{y}_{st+v}=(y_{st+v},...,y_{st+v-p+1},h_{st+v},...,h_{st+v-q+1})^{%
\prime }$ and $\underline{B}_{v}\left( \underline{\eta }_{t}\right) =(\alpha
_{0}(v)\eta _{st+v},0,...$ $,0,\alpha _{0}(v),0,...,0)^{\prime }$ as vectors
in $\mathbb{R}^{s}$, $\mathbb{R}^{d}$ and $\mathbb{R}^{d}$ respectively, and
set 
\begin{equation*}
\phi _{v}(\underline{\eta }_{t})=\left( 
\begin{array}{cc}
A_{v}(\underline{\eta }_{t}) & B_{v}(\underline{\eta }_{t}) \\ 
A_{v} & B_{v}%
\end{array}%
\right) _{d\times d}
\end{equation*}%
where%
\begin{equation*}
A_{v}(\underline{\eta }_{t})=\left( 
\begin{array}{cccc}
\alpha _{1}(v)\eta _{st+v} & \alpha _{2}(v)\eta _{st+v} & \ldots & \alpha
_{p}(v)\eta _{st+v} \\ 
1 & 0 & \ldots & 0 \\ 
0 & \ddots & \ddots & \vdots \\ 
0 & 0 & 1 & 0%
\end{array}%
\right) _{p\times p},B_{v}(\underline{\eta }_{t})=\left( 
\begin{array}{ccc}
\beta _{1}(v)\eta _{st+v} & \ldots & \beta _{q}(v)\eta _{st+v} \\ 
0 & \ldots & 0 \\ 
\vdots & \vdots & \vdots \\ 
0 & \cdots & 0%
\end{array}%
\right) _{p\times q}
\end{equation*}%
are$\ p\times p$ and $p\times q$ matrix valued polynomial functions of $%
\underline{\eta }_{t}$ and where%
\begin{equation*}
A_{v}=\left( 
\begin{array}{ccc}
\alpha _{1}(v) & \ldots & \alpha _{p}(v) \\ 
0 & \ldots & 0 \\ 
\vdots & \vdots & \vdots \\ 
0 & \cdots & 0%
\end{array}%
\right) _{q\times p},B_{v}=\left( 
\begin{array}{cccc}
\beta _{1}(v) & \beta _{2}(v) & \ldots & \beta _{q}(v) \\ 
1 & 0 & \ldots & 0 \\ 
0 & \ddots & \ddots & \vdots \\ 
0 & 0 & 1 & 0%
\end{array}%
\right) _{q\times q}
\end{equation*}%
this means that $A_{v}(\underline{x})$, $B_{v}(\underline{x})$ and $%
\underline{B}_{v}\left( \underline{x}\right) $ have entries\ and
coordinates, respectively, which are polynomial functions of the coordinates
of $\underline{x}$. Using the notations above, Equation $\left( \ref{Eq2_1}%
\right) $ can now be written as 
\begin{equation}
\underline{y}_{st+v}=\phi _{v}(\underline{\eta }_{t})\underline{y}_{st+v-1}+%
\underline{B}_{v}(\underline{\eta }_{t})  \tag{2.2}  \label{Eq2_2}
\end{equation}%
with $y_{st+v}=H^{\prime }\underline{y}_{st+v}$ where $H^{\prime }=\left(
1,0,...,0\right) _{1\times d}$. Equation $\left( \ref{Eq2_2}\right) $, is
the same as the defining equation for multivariate generalized periodic
autoregressive process introduced recently by Franses and Paap $(2000)$.
However, since Gladychev $(1961)$, with periodic time-varying coefficients,
it is possible to embed seasons\ into a multivariate stationary process.
More precisely $\left( \underline{Y}_{t}\right) _{t\in \mathbb{Z}}$\ where $%
\underline{Y}_{t}=\left( \underline{y}_{st+1}^{\prime },...,\underline{y}%
_{st+s}^{\prime }\right) ^{\prime }$ is a generalized autoregressive
process, i.e.,%
\begin{equation}
\underline{Y}_{t}=A(\underline{\eta }_{t})\underline{Y}_{t-1}+\underline{B}(%
\underline{\eta }_{t})  \tag{2.3}  \label{Eq2_3}
\end{equation}%
where $A(\underline{\eta }_{t})$ and $\underline{B}(\underline{\eta }_{t})$
are defined by blocks as

\begin{equation*}
A(\underline{\eta }_{t})=\left( 
\begin{array}{cccc}
O_{(d)} & \ldots & O_{(d)} & \phi _{1}(\underline{\eta }_{t}) \\ 
O_{(d)} & \ldots & O_{(d)} & \phi _{2}(\underline{\eta }_{t})\phi _{1}(%
\underline{\eta }_{t}) \\ 
\vdots & \vdots & \vdots & \vdots \\ 
O_{(d)} & \ldots & O_{(d)} & \prod\limits_{v=0}^{s-1}\phi _{s-v}(\underline{%
\eta }_{t})%
\end{array}%
\right) _{ds\times ds},\underline{B}(\underline{\eta }_{t})=\left( 
\begin{array}{c}
\underline{B}_{1}(\underline{\eta }_{t}) \\ 
\phi _{2}(\underline{\eta }_{t})\underline{B}_{1}(\underline{\eta }_{t})+%
\underline{B}_{2}(\underline{\eta }_{t}) \\ 
\vdots \\ 
\sum\limits_{k=1}^{s}\left\{ \prod\limits_{v=0}^{s-k-1}\phi _{s-v}(%
\underline{\eta }_{t})\right\} \underline{B}_{k}(\underline{\eta }_{t})%
\end{array}%
\right) _{ds\times 1}
\end{equation*}%
where, as usual, empty products are set equal to $I_{(d)}$.

In this section, we are interested in strictly stationary and causal
solutions for $\left( \ref{Eq1_1}\right) $. The results of this section are
based on theorems proved by Bougerol and Picard $(1992a)$ for generalized
autoregressive representation. Since $\left( \underline{\eta }_{t}\right)
_{t\in \mathbb{Z}}$ is an $i.i.d.$ process, $\left( A(\underline{\eta }_{t}),%
\underline{B}(\underline{\eta }_{t})\right) _{t\in \mathbb{Z}}$ is a
strictly stationary and ergodic process and since $E\left\{ \log
^{+}\left\Vert A(\underline{\eta }_{0})\right\Vert \right\} \leq $ $E\left\{
\left\Vert A(\underline{\eta }_{0})\right\Vert \right\} $ and $E\left\{ \log
^{+}\left\Vert \underline{B}(\underline{\eta }_{0})\right\Vert \right\} \leq
E\left\{ \left\Vert \underline{B}(\underline{\eta }_{0})\right\Vert \right\} 
$, then both $E\left\{ \log ^{+}\left\Vert A(\underline{\eta }%
_{0})\right\Vert \right\} $ and $E\left\{ \log ^{+}\left\Vert \underline{B}(%
\underline{\eta }_{0})\right\Vert \right\} $ are finite where for any $x>0$, 
$\log ^{+}x=\max \left( \log x,0\right) $.

\begin{theorem}
$\left[ Strict\text{ }stationary\text{ }solution\right] $ Equation $\left( %
\ref{Eq2_3}\right) $ has a unique strictly stationary and ergodic solution
if and only if the top Lyapunov exponent $\gamma _{L}\left( A\right) $
associated with the sequence matrices $\left( A(\underline{\eta }%
_{t})\right) _{t\in \mathbb{Z}}$%
\begin{equation}
\gamma _{L}\left( A\right) :=\inf_{t>0}\left\{ E\frac{1}{t}\log \left\Vert
\prod\limits_{j=0}^{t-1}A(\underline{\eta }_{t-j})\right\Vert \right\} 
\overset{a.s.}{=}\lim_{t\rightarrow \infty }\left\{ \frac{1}{t}\log
\left\Vert \prod\limits_{j=0}^{t-1}A(\underline{\eta }_{t-j})\right\Vert
\right\}  \tag{2.4}  \label{Lya}
\end{equation}%
is strictly negative. The unique stationary solution is causal, ergodic and
given by%
\begin{equation}
\underline{Y}_{t}=\sum\limits_{k=1}^{\infty }\left\{
\prod\limits_{j=0}^{k-1}A(\underline{\eta }_{t-j})\right\} \underline{B}(%
\underline{\eta }_{t-k})+\underline{B}(\underline{\eta }_{t})  \tag{2.5}
\label{Eq2_4}
\end{equation}%
where the above series converges almost surely $\left( a.s\right) $.
\end{theorem}

\begin{proof}
The theorem is a multidimensional\ extension of Theorem $1.3$ by Bougerol
and Picard $(1992a)$.
\end{proof}

\begin{remark}
\label{Remark_1}Since $\gamma _{L}\left( A\right) $ is independent of the
norm, thus we can work with some norms that make it rather straightforward
to show $\gamma _{L}\left( A\right) <0$. However, sufficient conditions
which ensure $\gamma _{L}\left( A\right) <0$ are

\begin{description}
\item[1.] $E\left\{ \log \left\Vert A(\underline{\eta }_{0})\right\Vert
\right\} <0$.

\item[2.] $E\left\{ \left\Vert \prod\limits_{j=0}^{t-1}A(\underline{\eta }%
_{t-j})\right\Vert ^{r}\right\} <1$ for some $r>0$ and $t\geq 1$, in which
case $E\left\{ \left\Vert \underline{Y}_{t}\right\Vert ^{r}\right\} <+\infty 
$.
\end{description}
\end{remark}

\noindent Now, a simple computation shows that%
\begin{equation*}
\prod\limits_{j=0}^{t}A(\underline{\eta }_{t-j})=A(\underline{\eta }%
_{t})\left( 
\begin{array}{cccc}
O_{(d)} & \ldots & O_{(d)} & O_{(d)} \\ 
O_{(d)} & \ldots & O_{(d)} & O_{(d)} \\ 
\vdots & \vdots & \vdots & \vdots \\ 
O_{(d)} & \ldots & O_{(d)} & \dprod\limits_{i=1}^{t-1}\left\{
\prod\limits_{v=0}^{s-1}\phi _{s-v}(\underline{\eta }_{t-i})\right\}%
\end{array}%
\right) .
\end{equation*}%
Therefore, because the top-Lyapunov\ exponent is independent of the norm, by
choosing a multiplicative norm it is straightforward to show that%
\begin{equation*}
\gamma _{L}\left( A\right) \leq \gamma _{L}(\Phi ):=\inf\limits_{t>0}\left\{
E\frac{1}{t}\log \left\Vert \dprod\limits_{i=1}^{t}\Phi \left( \underline{%
\eta }_{t-i}\right) \right\Vert \right\} .
\end{equation*}%
where $\Phi \left( \underline{\eta }_{t}\right) :=\left\{
\prod\limits_{v=0}^{s-1}\phi _{s-v}(\underline{\eta }_{t})\right\} $. We
have thus shown our first result which gives a sufficient condition for
strict stationarity.

\begin{theorem}
\label{Theorem2} Suppose that $\gamma _{L}(\Phi )<0$. Then for all $t\in 
\mathbb{Z}$ the series%
\begin{equation*}
\sum\limits_{k=0}^{\infty }\left\{ \prod\limits_{j=0}^{k-1}A(\underline{\eta 
}_{t-j})\right\} \underline{B}(\underline{\eta }_{t-k})
\end{equation*}%
converges $a.s.$and the process $\left( \underline{Y}_{t}\right) _{t\in 
\mathbb{Z}}$\ defined by $\left( \ref{Eq2_4}\right) $ is the unique strict
stationary, causal and ergodic solution of $\left( \ref{Eq2_3}\right) $.
\end{theorem}

\begin{example}
\label{Example1} In $PGARCH(1,1)$, writing $\phi _{v}\left( \underline{\eta }%
_{t}\right) =\left( \eta _{st+v},1\right) ^{\prime }\left( \alpha
_{1}(v),\beta _{1}(v)\right) $, we obtain%
\begin{equation*}
\dprod\limits_{v=0}^{s-1}\phi _{s-v}\left( \underline{\eta }_{t}\right)
=\left\{ \dprod\limits_{v=1}^{s-1}\left( \eta _{st+v}\alpha _{1}(v+1)+\beta
(v+1)\right) \right\} \left( \eta _{st+s},1\right) ^{\prime }\left( \alpha
_{1}(1),\beta _{1}(1)\right) .
\end{equation*}%
Hence%
\begin{equation*}
\log \left\Vert \dprod\limits_{i=1}^{t}\Phi \left( \underline{\eta }%
_{t-i}\right) \right\Vert
=\dsum\limits_{v=1}^{s-1}\dsum\limits_{i=1}^{t}\log \left( \eta
_{s(t-i)+v}\alpha _{1}(v+1)+\beta _{1}(v+1)\right) +\log \left\Vert
\dprod\limits_{i=1}^{t}\left( \eta _{s(t-i)+s},1\right) ^{\prime }\left(
\alpha _{1}(1),\beta _{1}(1)\right) \right\Vert \text{.}
\end{equation*}%
By the\ law of large numbers, a sufficient condition which ensure $\gamma
_{L}(A)<0$ is 
\begin{equation*}
\dsum\limits_{v=1}^{s}E\left\{ \log \left( \eta _{st+v}\alpha _{1}(v)+\beta
_{1}(v)\right) \right\} <0
\end{equation*}%
which reduces to the classical condition when $s=1$. It is worth noting that
the existence of explosive regimes (i.e., $E\left\{ \log \left( \eta
_{st+v}\alpha _{1}(v)+\beta _{1}(v)\right) \right\} >0$)\ does not preclude
(periodic) strict stationarity.
\end{example}

The top-Lyapunov exponent $\gamma _{L}(.)$ criterion seems difficult to
obtain explicitly, however a potential method to verify whether or not $%
\gamma _{L}(.)<0$ is via a Monte-Carlo simulation using Equation $\left( \ref%
{Eq2_3}\right) $. This fact heavily limits the interests of the criterion in
statistical applications. Indeed, the solution need to have some moments to
make an estimation theory possible and Condition $\left( \ref{Lya}\right) $
does not guarantee the existence of such moments. Therefore, we have to
search for conditions ensuring the existence of moments for the stationary
solution, for which, the top-Lyapunov exponent $\gamma _{L}(.)$ will be
automatically negative (see Remark $\left( \ref{Remark_2}\right) $).

\section{Necessary and sufficient second-order stationarity conditions}

In this section we examine the necessary and sufficient conditions ensuring
the existence of a unique causal periodically correlated $(PC)$ solution to $%
\left( \ref{Eq1_1}\right) $. This is equivalent to examining conditions of
the existence of solutions in $\mathbb{L}^{1}$ of the process $\left( 
\underline{Y}_{t}\right) _{t\in \mathbb{Z}}$ defined in equation $\left( \ref%
{Eq2_3}\right) $. Let $A=E\left\{ A(\underline{\eta }_{t})\right\} $ and $%
\underline{B}=E\left\{ \underline{B}(\underline{\eta }_{t})\right\} $ be the
expectations of the random matrix $A(\underline{\eta }_{t})$ and the random
vector \underline{$\underline{B}$}$(\underline{\eta }_{t})$, then $A$ and $%
\underline{B}$ are defined element-wise as 
\begin{equation*}
A=\left( 
\begin{array}{cccc}
O_{(d)} & \ldots & O_{(d)} & \phi _{1} \\ 
O_{(d)} & \ldots & O_{(d)} & \phi _{2}\phi _{1} \\ 
\vdots & \vdots & \vdots & \vdots \\ 
O_{(d)} & \ldots & O_{(d)} & \prod\limits_{v=0}^{s-1}\phi _{s-v}%
\end{array}%
\right) ,\underline{B}=\left( 
\begin{array}{c}
\underline{B}_{1} \\ 
\phi _{2}\underline{B}_{1}+\underline{B}_{2} \\ 
\vdots \\ 
\sum\limits_{k=1}^{s}\left\{ \prod\limits_{v=0}^{s-k-1}\phi _{s-v}\right\} 
\underline{B}_{k}%
\end{array}%
\right)
\end{equation*}%
where 
\begin{equation*}
\phi _{v}:=\left( 
\begin{array}{cc}
A_{v}(\mathbf{1}) & B_{v}(\mathbf{1}) \\ 
A_{v} & B_{v}%
\end{array}%
\right) _{d\times d},\underline{B}_{v}:=\left( 
\begin{array}{c}
\alpha _{0}(v) \\ 
O_{(p-1)\times 1} \\ 
\alpha _{0}(v) \\ 
O_{(q-1)\times 1}%
\end{array}%
\right) _{d\times 1}
\end{equation*}%
with $\mathbf{1}=(1,1,...,1)^{\prime }\in \mathbb{R}^{s}$. To verify that
the process $\left( \underline{Y}_{t}\right) _{t\in \mathbb{Z}}$ defined by $%
\left( \ref{Eq2_4}\right) $ is well-defined in $\mathbb{L}^{1}$, it is
sufficient to show that the coefficients $\underline{Y}_{t,k}=\left\{
\prod\limits_{j=0}^{k-1}A(\underline{\eta }_{t-j})\right\} \underline{B}(%
\underline{\eta }_{t-k})$ converge to zero in $\mathbb{L}^{1}$ (endowed with
any matrix norm) at an exponential rate, as $k\rightarrow \infty $.

\begin{lemma}
\label{Lemma1} Assume that%
\begin{equation}
\rho \left( \prod\limits_{v=0}^{s-1}\phi _{s-v}\right) <1  \tag{3.1}
\label{Eq3_1}
\end{equation}%
then the series $\sum\limits_{k=0}^{\infty }\left\{
\prod\limits_{j=0}^{k-1}A(\underline{\eta }_{t-j})\right\} \underline{B}(%
\underline{\eta }_{t-k})$ converges $a.s$. Furthermore, the process $\left( 
\underline{Y}_{t}\right) _{t\in \mathbb{Z}}$ defined by $\left( \ref{Eq2_4}%
\right) $ is stationary in $\mathbb{L}^{1}$ and satisfying $(\ref{Eq2_3})$.
\end{lemma}

\begin{proof}
First, we notice that $A(\underline{\eta }_{t})$ and $\underline{B}(%
\underline{\eta }_{t})$ are sequences of independent, non-negative random
matrices and vectors respectively, and $A(\underline{\eta }_{t-i})$ is
independent of $\underline{B}(\underline{\eta }_{t-j})$ for all $i\neq j$.
Therefore, we have $E\left\{ \prod\limits_{j=0}^{k-1}A(\underline{\eta }%
_{t-j})\right\} \underline{B}(\underline{\eta }_{t-k})=A^{k}\underline{B}.$
It can be shown that the characteristic polynomial of $A$ is $\det
(I_{(ds)}-\lambda A)=\det (I_{(s)}-\lambda \prod\limits_{v=0}^{s-1}\phi
_{s-v})$, hence $\rho \left( A\right) =\rho \left(
\prod\limits_{v=0}^{s-1}\phi _{s-v}\right) $. If $\rho \left(
\prod\limits_{v=0}^{s-1}\phi _{s-v}\right) <1$, then $\sum\limits_{k=1}^{%
\infty }A^{k}<+\infty ,$ and $\sum\limits_{k=0}^{\infty }E\left\{ \left\{
\prod\limits_{j=0}^{k-1}A(\underline{\eta }_{t-j})\right\} \underline{B}(%
\underline{\eta }_{t-k})\right\} <+\infty $ and thus $\sum\limits_{k=0}^{%
\infty }\left\{ \prod\limits_{j=0}^{k-1}A(\underline{\eta }_{t-j})\right\} 
\underline{B}(\underline{\eta }_{t-k})$ converges $a.s.$. This further
implies that $\left\{ \prod\limits_{j=0}^{k-1}A(\underline{\eta }%
_{t-j})\right\} \underline{B}(\underline{\eta }_{t-k})$ converges $a.s$. to
the zero matrix as $k\rightarrow \infty .$ It is obvious that the process $%
\left( \underline{Y}_{t}\right) _{t\in \mathbb{Z}}$ defined in $\left( \ref%
{Eq2_4}\right) $ is stationary in $\mathbb{L}^{1}$. Furthermore, we obtain 
\begin{eqnarray*}
\underline{Y}_{t} &=&\underline{B}(\underline{\eta }_{t})+A(\underline{\eta }%
_{t})\left\{ \underline{B}(\underline{\eta }_{t-1})+\sum\limits_{k=2}^{%
\infty }\left\{ \prod\limits_{j=0}^{k-1}A(\underline{\eta }_{t-j})\right\} 
\underline{B}(\underline{\eta }_{t-k})\right\} \\
&=&\underline{B}(\underline{\eta }_{t})+A(\underline{\eta }_{t})\left\{ 
\underline{B}(\underline{\eta }_{t-1})+\sum\limits_{k=1}^{\infty }\left\{
\prod\limits_{j=0}^{k-1}A(\underline{\eta }_{t-j-1})\right\} \underline{B}(%
\underline{\eta }_{t-k-1})\right\} \\
&=&\underline{B}(\underline{\eta }_{t})+A(\underline{\eta }_{t})\underline{Y}%
_{t-1}\text{.}
\end{eqnarray*}
\end{proof}

\begin{lemma}
\label{Lemma2} If $(\ref{Eq2_3})$ admits a stationary solution in $\mathbb{L}%
^{1}$, then $\rho \left( \prod\limits_{v=0}^{s-1}\phi _{s-v}\right) <1$.
Moreover, the stationary solution of $(\ref{Eq2_3})$ is unique, causal and
ergodic.
\end{lemma}

\begin{proof}
From $(\ref{Eq2_3})$, we obtain, by recursion, for any $n\geq 1$%
\begin{equation}
\underline{Y}_{t}=\underline{B}(\underline{\eta }_{t})+\sum\limits_{k=1}^{n}%
\left\{ \prod\limits_{j=0}^{k-1}A(\underline{\eta }_{t-j})\right\} 
\underline{B}(\underline{\eta }_{t-k})+\left\{ \prod\limits_{j=0}^{n}A(%
\underline{\eta }_{t-j})\right\} \underline{Y}_{t-n-1}.  \tag{3.2}
\label{Eq3_2}
\end{equation}%
Since, all $A(\underline{\eta }_{t}),$ $\underline{B}(\underline{\eta }_{t})$
and $\underline{Y}_{t}$ are non-negative, then 
\begin{equation*}
E\left\{ \underline{Y}_{t}\right\} \geq \sum\limits_{k=0}^{n}E\left\{
\left\{ \prod\limits_{j=0}^{k-1}A(\underline{\eta }_{t-j})\right\} 
\underline{B}(\underline{\eta }_{t-k})\right\} =\sum\limits_{k=0}^{n}A^{k}%
\underline{B}.
\end{equation*}%
This implies that $\sum\limits_{k=0}^{\infty }A^{k}\underline{B}<+\infty $.
Hence, $\lim\limits_{k\rightarrow \infty }A^{k}\underline{B}=0$. Since 
\begin{equation*}
A^{k}=\left( 
\begin{array}{cccc}
O_{(d)} & \ldots & O_{(d)} & \phi _{1}\left\{ \dprod\limits_{i=0}^{s-1}\phi
_{s-i}\right\} ^{k-1} \\ 
O_{(d)} & \ldots & O_{(d)} & \phi _{2}\phi _{1}\left\{
\dprod\limits_{i=0}^{s-1}\phi _{s-i}\right\} ^{k-1} \\ 
\vdots & \vdots & \vdots & \vdots \\ 
O_{(d)} & \ldots & O_{(d)} & \left\{ \dprod\limits_{i=0}^{s-1}\phi
_{s-i}\right\} ^{k}%
\end{array}%
\right) ,
\end{equation*}%
then to show that $\lim\limits_{k\rightarrow \infty }A^{k}=0$, it is
sufficient to show that $\lim\limits_{k\rightarrow \infty }\left\{
\dprod\limits_{i=0}^{s-1}\phi _{s-i}\right\} ^{k}=0$ which implies that $%
\rho \left( \dprod\limits_{i=0}^{s-1}\phi _{s-i}\right) =\rho \left(
A\right) <1$. To do this we shall prove that%
\begin{equation}
\lim\limits_{k\rightarrow \infty }\left\{ \dprod\limits_{i=0}^{s-1}\phi
_{s-i}\right\} ^{k}e_{i}=0\text{, }i=1,...,d  \tag{3.3}  \label{Eq3_3}
\end{equation}%
where $\left( e_{v}\right) _{1\leq v\leq d}$ is the canonical basis of $%
\mathbb{R}^{d}$, i.e., $e_{v}=\left( \delta _{1,v},\delta _{2,v},...,\delta
_{d,v}\right) ^{\prime }$ where $\delta _{i,j}=1$ if $i=j$ and $0$
otherwise.\ Because the $d\times d$ $j-th$ block of vector$\ A^{k}\underline{%
B}$ is $\left( A^{k}\underline{B}\right) _{j}=\left\{
\dprod\limits_{v=0}^{j-1}\phi _{j-v}\right\} \dsum\limits_{l=1}^{s}\left\{
\dprod\limits_{i=0}^{s-1}\phi _{s-i}\right\} ^{k-1}\left\{
\dprod\limits_{v=0}^{s-l-1}\phi _{s-v}\right\} \underline{B}_{l}$, $%
j=1,...,s $ we deduce that for $j=s$, 
\begin{equation}
\lim\limits_{k\rightarrow \infty }\left\{ \dprod\limits_{i=0}^{s-1}\phi
_{s-i}\right\} ^{k}\left\{ \dprod\limits_{v=0}^{s-l-1}\phi _{s-v}\right\} 
\underline{B}_{l}=0,\text{ \ \ }l=1,...,s.  \tag{3.4}  \label{Eq3_4}
\end{equation}%
Since $\underline{B}_{v}=\alpha _{0}(v)e_{1}+\alpha _{0}(v)e_{p+1}$ for $%
v=1,...,s$ and $\alpha _{0}(v)>0$, we obtain from $\left( \ref{Eq3_4}\right) 
$ for \ $l=1,...,s$%
\begin{equation*}
\lim_{k\rightarrow \infty }\left\{ \dprod\limits_{i=0}^{s-1}\phi
_{s-i}\right\} ^{k}\left\{ \dprod\limits_{v=0}^{s-l-1}\phi _{s-v}\right\}
e_{1}=0,\lim_{k\rightarrow \infty }\left\{ \dprod\limits_{i=0}^{s-1}\phi
_{s-i}\right\} ^{k}\left\{ \dprod\limits_{v=0}^{s-l-1}\phi _{s-v}\right\}
e_{p+1}=0
\end{equation*}%
and we use the relationships $\phi _{v}e_{i}=\alpha _{i}(v)\left(
e_{1}+e_{p+1}\right) +e_{i+1}$ for $i=1,...,p$ and $\phi _{v}e_{p+i}=\beta
_{i}(v)\left( e_{1}+e_{p+1}\right) +e_{p+i+1}$ for $i=1,...,q-1$ we obtain%
\begin{eqnarray*}
0 &=&\lim_{k\rightarrow \infty }\left\{ \dprod\limits_{i=0}^{s-1}\phi
_{s-i}\right\} ^{k}\left\{ \dprod\limits_{v=0}^{s-l-1}\phi _{s-v}\right\}
e_{1}=\lim_{k\rightarrow \infty }\left\{ \dprod\limits_{i=0}^{s-1}\phi
_{s-i}\right\} ^{k}e_{s+1-l}\text{, }l=1,...,s \\
0 &=&\lim_{k\rightarrow \infty }\left\{ \dprod\limits_{i=0}^{s-1}\phi
_{s-i}\right\} ^{k}\left\{ \dprod\limits_{v=0}^{s-l-1}\phi _{s-v}\right\}
e_{p+1}=\lim_{k\rightarrow \infty }\left\{ \dprod\limits_{i=0}^{s-1}\phi
_{s-i}\right\} ^{k}e_{p+s+1-l}\text{, }l=1,...,s
\end{eqnarray*}%
and hence $\left( \ref{Eq3_3}\right) $ holds for $i=1,...s$ and $%
i=p+1,...p+s $. On the other hand, from $\left( \ref{Eq3_4}\right) $ we have%
\begin{eqnarray*}
0 &=&\lim\limits_{k\rightarrow \infty }\left\{ \dprod\limits_{i=0}^{s-1}\phi
_{s-i}\right\} ^{k}e_{i}=\lim\limits_{k\rightarrow \infty }\left\{
\dprod\limits_{i=0}^{s-1}\phi _{s-i}\right\} ^{k-1}e_{i+s}\text{, }i=1,...,s
\\
0 &=&\lim\limits_{k\rightarrow \infty }\left\{ \dprod\limits_{i=0}^{s-1}\phi
_{s-i}\right\} ^{k}e_{p+i}=\lim\limits_{k\rightarrow \infty }\left\{
\dprod\limits_{i=0}^{s-1}\phi _{s-i}\right\} ^{k-1}e_{p+i+1}\text{, }%
i=1,...,s
\end{eqnarray*}%
This concludes the proof of the lemma. To prove the uniqueness, let $\left( 
\underline{Z}_{t}\right) _{t\in \mathbb{Z}}$ be another stationary process
conforming to $(\ref{Eq2_3})$. Then $\left( \underline{Z}_{t}\right) _{t\in 
\mathbb{Z}}$ satisfies an equation similar to $(\ref{Eq3_2})$. By setting $%
\underline{W}_{t}=\underline{Z}_{t}-\underline{Y}_{t}$, we obtain, $\forall
m\geq 1:\underline{W}_{t}=\prod\limits_{j=0}^{m}A(\underline{\eta }_{t-j})%
\underline{W}_{t-m-1}.$ Defining $\underline{W}=E\left\{ \underline{W}%
_{t}\right\} $, we have $\underline{W}=A^{m+1}\underline{W}$, since $\lambda
:=\rho \left( A\right) <1$ we conclude that $\underline{W}=0$ (with
probability $1$). Hence the uniqueness follows.
\end{proof}

\begin{theorem}
\label{Theorem3} A necessary and sufficient condition for the existence of
unique stationary solution in $\mathbb{L}^{1}$ of equation $(\ref{Eq2_3})$
is that $\left( 3.1\right) $ holds. Moreover, this stationary solution is
causal and ergodic.
\end{theorem}

\begin{proof}
The proof follows from Lemma \ref{Lemma1} and \ref{Lemma2}. The fact that
stationary solution is ergodic is obtained by the same argument as in
Bougerol and Picard $(1992a)$.
\end{proof}

\begin{remark}
\label{Remark_2}Since $\gamma _{L}\left( \Phi \right) <\log \rho \left(
\prod\limits_{v=0}^{s-1}\phi _{s-v}\right) $ (see Kesten and Spitzer, $1984$%
), the condition $\left( \ref{Eq3_1}\right) $ is necessary and sufficient
for the existence of a strictly stationary solution in $\mathbb{L}^{1}$ of $%
\left( \ref{Eq2_3}\right) $.
\end{remark}

\section{Existence of the higher-order moments}

In this section, we present a necessary and sufficient conditions for the
existence of finite higher order moments for a $PGARCH$ process.

\begin{theorem}
\label{Theorem4} Let $\left( \underline{Y}_{t}\right) _{t\geq 0}$ be the
stationary solution of model $\left( \ref{Eq2_3}\right) $. Assume $\kappa
_{2}<+\infty $ where $\kappa _{k}=:E\left\{ \varepsilon _{t}^{2k}\right\} $.

\begin{description}
\item[1.] If%
\begin{equation}
\rho \left( \prod\limits_{v=0}^{s-1}E\left\{ \phi _{s-v}^{\otimes 2}(%
\underline{\eta }_{t})\right\} \right) <1  \tag{4.1}  \label{Cond1}
\end{equation}%
then $\underline{Y}_{t}\in \mathbb{L}^{2}$.

\item[2.] Conversely, if $\rho \left( \prod\limits_{v=0}^{s-1}E\left\{ \phi
_{s-v}^{\otimes 2}(\underline{\eta }_{t})\right\} \right) \geq 1$, then
there is no strictly stationary solution $\left( \underline{Y}_{t}\right)
_{t\in \mathbb{Z}}$ to model $\left( \ref{Eq2_3}\right) $ such that $%
E\left\{ \underline{Y}_{t}^{\otimes 2}\right\} <+\infty .$
\end{description}
\end{theorem}

\begin{proof}

\begin{description}
\item[1.] We first define the following $\mathbb{R}^{sd}-$valued stochastic
processes 
\begin{equation*}
\underline{S}_{n}(t)=\left\{ 
\begin{array}{l}
0\text{ \ \ \ \ \ \ \ \ \ \ \ \ \ \ \ \ \ \ \ \ \ \ \ \ \ \ \ \ \ \ \ \ \ \
\ \ \ if }n<0 \\ 
\underline{B}(\underline{\eta }_{t})+A(\underline{\eta }_{t})\underline{S}%
_{n-1}(t-1)\text{ \ \ \ \ \ if }n\geq 0,%
\end{array}%
\right.
\end{equation*}%
and for all $n\in \mathbb{Z}$: $\underline{\Delta }_{n}(t)=\underline{S}%
_{n}(t)-\underline{S}_{n-1}(t)$. It can be easily shown that, for all $n$ $%
\geq 0$, $\underline{S}_{n}(t)$ and $\underline{\Delta }_{n}(t)$ are
measurable functions of $\underline{\eta }_{t},\underline{\eta }_{t-1},...,%
\underline{\eta }_{t-n}$. Hence the processes $\left( \underline{S}%
_{n}(t)\right) _{t\in \mathbb{Z}}$ and $\left( \underline{\Delta }%
_{n}(t)\right) _{t\in \mathbb{Z}}$ are stationary. From the definition of $%
\underline{S}_{n}(t)$ and $\underline{\Delta }_{n}(t)$, we can verify that 
\begin{equation*}
\underline{\Delta }_{n}(t)=\left\{ 
\begin{array}{l}
0\text{ \ \ \ \ \ \ \ \ \ \ \ \ \ \ \ \ \ \ \ \ \ \ \ \ \ \ \ \ \ \ \ \ \ \
\ \ \ \ \ if }n<0 \\ 
\underline{B}(\underline{\eta }_{t})\text{ \ \ \ \ \ \ \ \ \ \ \ \ \ \ \ \ \
\ \ \ \ \ \ \ \ \ \ \ \ \ \ \ \ if }n=0 \\ 
A(\underline{\eta }_{t})\underline{\Delta }_{n-1}(t-1)\text{ \ \ \ \ \ \ \ \
\ \ \ \ \ \ \ \ \ if }n>0%
\end{array}%
\right.
\end{equation*}%
\noindent For all $n\in \mathbb{Z}$, define $\Gamma _{n}^{(2)}(t)=\underline{%
\Delta }_{n}^{\otimes 2}(t)=Vec\left\{ \underline{\Delta }_{n}(t)\underline{%
\Delta }_{n}^{\prime }(t)\right\} $. Using the properties of Kronecker
product, we obtain $\Gamma _{n}^{(2)}(t)=A^{\otimes 2}(\underline{\eta }%
_{t})\Gamma _{n-1}^{(2)}(t-1)$ for $n>0$ and $E\left\{ \Gamma
_{n}^{(2)}(t)\right\} =\left( A^{(2)}\right) ^{n}E\left\{ \Gamma
_{0}^{(2)}(t-n)\right\} =\left( A^{(2)}\right) ^{n}E\left\{ \underline{B}%
^{\otimes 2}(\underline{\eta }_{t-n})\right\} $, where $A^{(2)}:=E$ $\left\{
A^{\otimes 2}\left( \underline{\eta }_{t})\right) \right\} .$ Since $\rho
\left( A^{(2)}\right) =\rho \left( \prod\limits_{v=0}^{s-1}E\left\{ \phi
_{s-v}^{\otimes 2}(\underline{\eta }_{t})\right\} \right) <1$, we conclude
that $\underline{S}_{n}(t)$ converges in $\mathbb{L}^{2}$ and almost surely
to some limit $\underline{Y}_{t}\in \mathbb{L}^{2}$ which is the solution of
equation $(\ref{Eq2_3})$. This completes the proof.

\item[2.] From $\left( \ref{Eq3_2}\right) $ we obtain%
\begin{equation*}
E\left\{ \underline{Y}_{t}^{\otimes 2}\right\} \geq
\sum\limits_{k=0}^{\infty }E\left\{ \left\{ \prod\limits_{j=0}^{k-1}A(%
\underline{\eta }_{t-j})\right\} \underline{B}(\underline{\eta }%
_{t-k})\right\} ^{\otimes 2}=\sum\limits_{k=0}^{\infty }\left(
A^{(2)}\right) ^{k}\underline{B}^{(2)}
\end{equation*}%
where $\underline{B}^{(2)}:=E\left\{ \underline{B}^{\otimes 2}(\underline{%
\eta }_{t})\right\} $ and the conclusion follows.
\end{description}
\end{proof}

\noindent The result of the above theorem can be further extended to the
higher-order moments.

\begin{theorem}
\label{Theorem5} Let $\left( \underline{Y}_{t}\right) _{t\geq 0}$ be the
stationary solution of model $\left( \ref{Eq2_3}\right) $. Assume that $%
\kappa _{2(r-1)}<+\infty $ where $r>2$.

\begin{description}
\item[1.] If $\rho \left( \prod\limits_{v=0}^{s-1}E\left\{ \phi
_{s-v}^{\otimes r}(\underline{\eta }_{t})\right\} \right) <1$ then $%
\underline{Y}_{t}\in \mathbb{L}^{r}$.

\item[2.] Conversely, if $\rho \left( \prod\limits_{v=0}^{s-1}E\left\{ \phi
_{s-v}^{\otimes r}(\underline{\eta }_{t})\right\} \right) \geq 1$, then
there is no strictly stationary solution $\left( \underline{Y}_{t}\right)
_{t\in \mathbb{Z}}$ to model $\left( \ref{Eq2_3}\right) $ such that $%
E\left\{ \underline{Y}_{t}^{\otimes r}\right\} <+\infty .$
\end{description}
\end{theorem}

\begin{proof}
Define $\underline{S}_{n}(t)$ and $\underline{\Delta }_{n}(t)$ as in Theorem %
\ref{Theorem4} and let $\Gamma _{n}^{(r)}(t)=\underline{\Delta }%
_{n}^{\otimes r}(t)$. Since $\Gamma _{n}^{(r)}(t)=A^{\otimes r}(\underline{%
\eta }_{t})\Gamma _{n-1}^{(r)}(t-1)$ for $n>0$, the proof is similar to that
for the Theorem \ref{Theorem4} and thus we omit the details.
\end{proof}

\section{Covariance structure}

To get the covariance structure of the squared $PGARCH$ process, we first
assume that the Condition $\left( \ref{Cond1}\right) $ holds, this implies
that $\left( \ref{Eq2_2}\right) $ has a unique (in $\mathbb{L}^{2}$ sense) $%
PC$ solution. Taking expectation on both sides of $\left( \ref{Eq2_2}\right) 
$ and using the notation $\underline{\mu }_{k}(v)=E\left\{ \underline{y}%
_{st+v}^{\otimes k}\right\} $, $k=1,2$ gives%
\begin{equation}
\underline{\mu }_{1}(v)=\phi _{v}\underline{\mu }_{1}(v-1)+\underline{B}_{v}%
\text{, }v=1,...,s.  \tag{5.1}  \label{Mom1}
\end{equation}%
Iterating $\left( \ref{Mom1}\right) $ $s-$ times and requiring $\underline{%
\mu }_{k}(0)=\underline{\mu }_{k}(s)$, we obtain%
\begin{equation*}
\underline{\mu }_{1}(s)=\left( I_{(d)}-\dprod\limits_{v=0}^{s-1}\phi
_{s-v}\right) ^{-1}\dsum\limits_{j=0}^{s-1}\left\{
\dprod\limits_{v=0}^{j-1}\phi _{s-v}\right\} \underline{B}_{s-j}.
\end{equation*}%
Further manipulations in $\left( \ref{Mom1}\right) $ provide the desired
seasonal mean%
\begin{equation*}
\underline{\mu }_{1}(v)=\dsum\limits_{j=0}^{v-1}\left\{
\dprod\limits_{i=0}^{j-1}\phi _{v-i}\right\} \underline{B}_{v-j}+\left\{
\dprod\limits_{i=0}^{v-1}\phi _{s-i}\right\} \underline{\mu }_{1}(s),\text{ }%
v=1,...,s.
\end{equation*}%
The seasonal variance can be obtained as follow 
\begin{eqnarray*}
\underline{\mu }_{2}(v) &=&E\left\{ \left( \phi _{v}(\underline{\eta }_{t})%
\underline{y}_{st+v-1}+\underline{B}_{v}(\underline{\eta }_{t})\right)
\otimes \left( \phi _{v}(\underline{\eta }_{t})\underline{y}_{st+v-1}+%
\underline{B}_{v}(\underline{\eta }_{t})\right) \right\} \\
&=&\phi _{v}^{(2)}\underline{\mu }_{2}(v-1)+\underline{\varphi }_{v}
\end{eqnarray*}%
where\ $\underline{\varphi }_{v}=\underline{B}_{v}^{(2)}+\left( E\left\{
\phi _{v}(\underline{\eta }_{t})\otimes \underline{B}_{v}(\underline{\eta }%
_{t})+\underline{B}_{v}(\underline{\eta }_{t})\otimes \phi _{v}(\underline{%
\eta }_{t})\right\} \right) \underline{\mu }_{1}(v-1)$ with $\phi
_{v}^{(2)}:=E\left\{ \phi _{v}^{\otimes 2}(\underline{\eta }_{t})\right\} $
and $\underline{B}_{v}^{(2)}=E\left\{ \underline{B}_{v}^{\otimes 2}(%
\underline{\eta }_{t})\right\} $. Thus 
\begin{eqnarray*}
\underline{\mu }_{2}(s) &=&\left( I_{(d^{2})}-\dprod\limits_{v=0}^{s-1}\phi
_{s-v}^{(2)}\right) ^{-1}\dsum\limits_{j=0}^{s-1}\left\{
\dprod\limits_{v=0}^{j-1}\phi _{s-v}^{(2)}\right\} \underline{\varphi }%
_{s-j}, \\
\underline{\mu }_{2}(v) &=&\dsum\limits_{j=0}^{v-1}\left\{
\dprod\limits_{i=0}^{j-1}\phi _{v-i}^{(2)}\right\} \underline{\varphi }%
_{v-j}+\left\{ \dprod\limits_{i=0}^{v-1}\phi _{s-i}^{(2)}\right\} \underline{%
\mu }_{2}(s).
\end{eqnarray*}%
Now, noting that for any $h>0$, we have%
\begin{eqnarray*}
\underline{\gamma }_{v}(h) &=&E\left\{ \underline{y}_{st+v}\otimes 
\underline{y}_{st+v-h}\right\} \\
&=&E\left\{ \phi _{v}(\underline{\eta }_{t})\otimes I_{(d)}\right\} E\left\{ 
\underline{y}_{st+v-1}\otimes \underline{y}_{st+v-h}\right\} +\underline{B}%
_{v}\otimes \underline{\mu }_{1}(v-h) \\
&=&\left( \phi _{v}\otimes I_{(d)}\right) \underline{\gamma }_{v-1}(h-1)+%
\underline{B}_{v}\otimes \underline{\mu }_{1}(v-h) \\
&=&\dsum\limits_{k=0}^{h-1}\left\{ \dprod\limits_{i=0}^{k-1}\left( \phi
_{v-i}\otimes I_{(d)}\right) \right\} \underline{B}_{v-k}\otimes \underline{%
\mu }_{1}(v-h-k)+\left\{ \dprod\limits_{i=0}^{h-1}\left( \phi _{v-i}\otimes
I_{(d)}\right) \right\} \underline{\mu }_{2}(v-h).
\end{eqnarray*}%
In the above equations, $\phi _{v}$, $\underline{\mu }_{1}(v)$ and $%
\underline{\mu }_{2}(v)$ are interpreted periodically in $v$ with period $s$.

\begin{example}
\label{Example2}The simplest, but often very useful $PGARCH$ process is the $%
PGARCH(1,1)$ process. Further computations show that $\left( \ref{Eq3_1}%
\right) $ and $\left( \ref{Cond1}\right) $ reduce to $\prod\limits_{v=1}^{s}%
\theta _{1}(v)<1$ and $\prod\limits_{v=1}^{s}\theta _{2}(v)<1$\ respectively
where $\theta _{1}(v)=\alpha _{1}(v)+\beta _{1}(v)$ and where $\theta
_{2}(v)=\kappa _{2}\alpha _{1}^{2}(v)+\beta _{1}^{2}(v)+2\alpha _{1}(v)\beta
_{1}(v)$ which we henceforth assume. However, it is worth noting that the
existence of explosive regimes ( i.e. $\theta _{1}(v)>1$) does not preclude
the periodic stationarity. The seasonal moments are%
\begin{eqnarray*}
\mu _{1}(s) &=&\left( 1-\dprod\limits_{v=1}^{s}\theta _{1}(v)\right)
^{-1}\dsum\limits_{j=0}^{s-1}\left\{ \dprod\limits_{v=0}^{j-1}\theta
_{1}(s-v)\right\} \alpha _{0}(s-j) \\
\mu _{1}(v) &=&\dsum\limits_{j=0}^{v-1}\left\{
\dprod\limits_{i=0}^{j-1}\theta _{1}(v-i)\right\} \alpha _{0}(v-j)+\left\{
\dprod\limits_{i=0}^{v-1}\theta _{1}(v-i)\right\} \mu _{1}(s) \\
\mu _{2}(s) &=&\kappa _{2}\left( 1-\dprod\limits_{v=1}^{s}\theta
_{2}(v)\right) ^{-1}\dsum\limits_{j=0}^{s-1}\left\{
\dprod\limits_{v=0}^{j-1}\theta _{2}(s-v)\right\} \left\{ \alpha
_{0}^{2}(s-j)+2\alpha _{0}(s-j)\theta _{1}(s-v)\mu _{1}(s-j-1)\right\} \\
\mu _{2}(v) &=&\dsum\limits_{j=0}^{v-1}\left\{
\dprod\limits_{i=0}^{j-1}\theta _{2}(v-i)\right\} \left\{ \alpha
_{0}^{2}(v-j)+2\alpha _{0}(v-j)\theta _{1}(v-j)\mu _{1}(v-j-1)\right\}
+\left\{ \dprod\limits_{i=0}^{v-1}\theta _{2}(v-i)\right\} \mu _{2}(s).
\end{eqnarray*}%
The covariance structure of $PGARCH(1,1)$ process can be obtained as%
\begin{equation*}
\gamma _{v}(h)=\left\{ 
\begin{array}{l}
\alpha _{0}(v)\mu _{1}(v-1)+\theta _{1}(v)\mu _{2}(v-1)\text{, }h=1 \\ 
\alpha _{0}(v)\mu _{1}(v-1)+\theta _{1}(v)\gamma _{v-1}(h-1)\text{, }h\geq 2.%
\end{array}%
\right.
\end{equation*}%
When compared to the covariance function of second order $GARCH(1,1)$
process, the above formulas are quite complex. On can verify that these
expressions reduce to the classical $GARCH(1,1)$ forms when the\ $%
PGARCH(1,1) $ parameters are constant in $v$. In general, calculation in $%
PGARCH$ process are very tedious.
\end{example}

\section{Geometric ergodicity and strong mixing}

The basic tools presented here are drawn from \ the monograph by Doukhan $%
(1994)$. In this section we analyze the statistical properties of $PGARCH$
process, such as the geometric ergodicity and the strong mixing. These
concepts are fundamental in central limit theorem and law of large numbers\
which can be employed\ to derive asymptotic normality, consistency of
maximum likelihood style estimators and inference with the model.

\begin{definition}
Let $\left( \underline{X}_{t}\right) _{t\geq 0}$ be a discrete Markov chain
taking values in $\mathbb{R}^{k}$ $k\geq 1$ with time homogeneous $t$-step
transition probabilities i.e., $P^{t}\left( \underline{x},A\right) =P\left( 
\underline{X}_{t}\in A|\underline{X}_{0}=\underline{x}\right) $ where $%
\underline{x}\in \mathbb{R}^{k}$, $A\in $\QTR{cal}{B}$_{\mathbb{R}^{k}}%
\mathcal{\ }$and where \QTR{cal}{B}$_{\mathbb{R}^{k}}$ is a Borel $\sigma -$%
field on $\mathbb{R}^{k}$ with $P^{1}\left( \underline{x},A\right) =P\left( 
\underline{x},A\right) .$ Let $\pi $ be the invariant probability measure on 
$\left( \mathbb{R}^{k},\mathcal{B}_{\mathbb{R}^{k}}\right) \mathcal{\ }$,
i.e.,$\forall $ $A\in $\QTR{cal}{B}$_{\mathbb{R}^{k}}$: $\pi (A)=\dint P(%
\underline{x},A)\pi (d\underline{x})$.

\begin{description}
\item[a.] The chain $\left( \underline{X}_{t}\right) _{t\geq 0}$ is said to
be geometrically ergodic, if there exists some $0<r<1$ such that%
\begin{equation}
\forall \underline{x}\in \mathbb{R}^{k}\text{, }\left\Vert P^{t}\left(
x,.\right) -\pi (.)\right\Vert _{V}=o(r^{t})  \tag{6.1}  \label{Eq6_1}
\end{equation}%
where $\left\Vert .\right\Vert _{V}$ is the total variation norm.
Furthermore, if the chain $\left( \underline{X}_{t}\right) _{t\in \mathbb{Z}%
} $ is started with an initial distribution $\pi $, the process is strictly
stationary.

\item[b.] The $\beta -$mixing coefficients\ are defined by%
\begin{equation*}
\beta _{X}(k)=E\left\{ \sup_{B\in \sigma \left( \underline{X}_{t},t\geq
k\right) }\left\vert P(B|\sigma \left( \underline{X}_{t},t\leq 0\right)
)-P(B)\right\vert \right\} =\frac{1}{2}\sup
\dsum\limits_{i=1}^{I}\dsum\limits_{j=1}^{J}\left\vert P(A_{i}\cap
B_{j})-P(A_{i})P(B_{j})\right\vert
\end{equation*}%
where the supremum of the last equality is taken over all pairs of
partitions $\left\{ A_{1},...,A_{I}\right\} $ and $\left\{
B_{1},...,B_{I}\right\} $ of $\Omega $ such that $A_{i}\in \sigma \left( 
\underline{X}_{t},t\leq 0\right) $ for all $i$ and $B_{j}\in \sigma \left( 
\underline{X}_{t},t\geq k\right) $ for all $j$. The process $\left( 
\underline{X}_{t}\right) _{t\in \mathbb{Z}}$ is called $\beta $-mixing if $%
\lim\limits_{k\rightarrow \infty }$ $\beta _{X}(k)=0$.
\end{description}
\end{definition}

\noindent This powerful result helps us to develop asymptotic results, since
convergence in distribution is ensured for all measurable functions of the
chain. Perhaps one of the most well-known criterion used in establishing the
geometric ergodicity of a Markov chain is the drift condition\ developed in
Tweedie $(1975)$ and employed for the analysis of stochastic stability.

\begin{definition}
A drift function $g:\mathbb{R}^{k}\rightarrow \left[ 1,\infty \right[ $
satisfies a $1$-step geometrical drift criterion (relative to a Markov
chain) if there exists a compact $K\subset \mathbb{R}^{k}$ and a positive
constants $b$ and $0<\lambda <1$,%
\begin{equation*}
E\left\{ g(\underline{X}_{t})|\underline{X}_{t-1}=\underline{x}\right\} \leq
\left\{ 
\begin{array}{l}
\lambda g(\underline{x})\text{ if }\underline{x}\notin K \\ 
b\text{ if }\underline{x}\in K%
\end{array}%
\right.
\end{equation*}%
where $g$ is interpreted as a generalized energy function and the compact
set $K$ as the centre of attraction.
\end{definition}

\begin{remark}
One consequence of the geometric ergodicity is that the stationary Markov
chain $\left( \underline{X}_{t}\right) _{t\in \mathbb{Z}}$ is $\beta $%
-mixing, and hence strongly mixing, with geometric rate. Indeed, Davidov $%
(1973)$ showed that for an ergodic Markov chain $\left( \underline{X}%
_{t}\right) _{t\in \mathbb{Z}}$ with invariant probability measure $\pi $, \ 
$\beta _{X}(k)=\dint \left\Vert P^{k}\left( x,.\right) -\pi (.)\right\Vert
_{V}\pi (d\underline{x})$. Thus $\beta _{X}(k)=o\left( \rho ^{k}\right) $ if 
$\left( \ref{Eq6_1}\right) $ holds.
\end{remark}

\noindent In what follows, we shall assume, without loss of generality, that 
$p=q$, otherwise zeros can be filled. To derive the geometric ergodicity and 
$\beta -$mixing results, we must assume the following conditions for the
sequence $\left( \underline{\eta }_{t}\right) _{t\in \mathbb{Z}}$.

\begin{condition0}
\label{B.0} The i.i.d. sequence $\left( \underline{\eta }_{t}\right) _{t\in 
\mathbb{Z}}$ has a probability distribution function absolutely\ continuous
with respect to the Lebesgue measure and such that the density function take
positive values almost surely on its support $E\subset $ $\mathbb{R}^{s}$.
\end{condition0}

\noindent We are now ready to state the main result of this section which
establishes statistical properties of $PGARCH$ processes, in particular, we
focus on the $\beta -$mixing property with exponential decay.

\begin{theorem}
\label{Main_R}Under Conditions $\left( \mathbf{B.0}\right) $ and $\left( \ref%
{Eq3_1}\right) $ and assume that

\begin{description}
\item[1.] $\rho \left( \dprod\limits_{v=1}^{s}B_{v}\right) <1$

\item[2.] There are $0<r\leq 1$ and $\rho <1$ such that $E\left\{ \left\Vert
A(\underline{\eta }_{0})\right\Vert ^{r}\right\} \leq \rho $ and $\underline{%
\eta }_{0}\in \mathbb{L}^{r}$.
\end{description}

\noindent then the process $\left( \underline{Y}_{t}\right) _{t\in \mathbb{Z}%
}$ defined by $\left( \ref{Eq2_3}\right) $ is geometrically ergodic.
Moreover, if initialized from its invariant measure, $\left( \underline{Y}%
_{t}\right) _{t\in \mathbb{Z}}$ is strictly stationary and $\beta -$mixing
with exponential decay.
\end{theorem}

\begin{remark}
The condition $E\left\{ \left\Vert A(\underline{\eta }_{0})\right\Vert
^{r}\right\} <1$ for some $0<r\leq 1$ in some neighborhood of zero is
satisfied\ if $E\left\{ \log \left\Vert A(\underline{\eta }_{0})\right\Vert
\right\} <0$ and $E\left\{ \left\Vert A(\underline{\eta }_{0})\right\Vert
^{\epsilon }\right\} <\infty $ for some $\epsilon >0$. Indeed, the function $%
\vartheta (v)=E\left\{ \left\Vert A(\underline{\eta }_{0})\right\Vert
^{v}\right\} $ has derivative $\vartheta ^{\prime }(0)=E\left\{ \log
\left\Vert A(\underline{\eta }_{0})\right\Vert \right\} <0$, hence $%
\vartheta (v)$ decreases in small neighborhood of zero, and since $\vartheta
(0)=1$ it follows that $\vartheta (r)<1$ for $0<$ $r\leq 1$. On the other
hand $E\left\{ \left\Vert A(\underline{\eta }_{0})\right\Vert ^{r}\right\}
<1 $ implies that $E\left\{ \log \left\Vert A(\underline{\eta }%
_{0})\right\Vert \right\} <0$ by Jensen's inequality.
\end{remark}

\begin{proof}
The proof is based on Carrasco and Chen $(2002,$ Theorem $1)$. Assumption $%
\mathbf{B.0}$ ensures that Carrasco and Chen's ($2002$ p. 20) condition on
the innovation term holds. We must then prove that $A\left( \underline{\eta }%
_{t}\right) $ and $\underline{B}\left( \underline{\eta }_{t}\right) $
satisfy their assumptions $\left( A_{0}\right) $, $\left( A_{1}\right) $, $%
\left( A_{2}^{\ast }\right) $ and $\left( A_{3}^{\ast }\right) $

\begin{condition1}
\label{A.0} Both $A\left( \underline{\eta }_{t}\right) $ and $\underline{B}%
\left( \underline{\eta }_{t}\right) $\ are polynomial functions of $%
\underline{\eta }_{t}$, therefore the measurability condition is trivially
satisfied.
\end{condition1}

\begin{condition2}
\label{A.1} We must show that $\rho \left( A(0)\right) <1$. It turns out
however that the nonzero eigenvalues of $A(0)$ are the nonzero eigenvalues
of $\Phi (0)$. It is not hard to see that then $\rho \left( \Phi (0)\right)
=\rho \left( \dprod\limits_{v=1}^{s}B_{v}\right) $ (see Horn and Johnson, $%
1999$, p. 68).
\end{condition2}

\begin{condition3}
\label{A*.2} That $\dsum\limits_{k\geq 0}\left\{
\dprod\limits_{j=0}^{k-1}A\left( \underline{\eta }_{t-j}\right) \right\} 
\underline{B}\left( \underline{\eta }_{t-k}\right) $ converges almost surely
to some constant immediately follows from Lemma $\ref{Lemma1}$. It remains
to show that $\left\{ \dprod\limits_{j=0}^{k-1}A\left( \underline{\eta }%
_{t-j}\right) \right\} X$ converges almost surely to zero for all $X\in 
\mathbb{R}^{ds}$ as $k\rightarrow \infty $.\ As $X$ is nonrandom,\ it
suffices to show that $\left\{ \dprod\limits_{j=0}^{k-1}A\left( \underline{%
\eta }_{t-j}\right) \right\} $ converges almost surely to zero as $%
k\rightarrow \infty $. Since $E\left\{ \dprod\limits_{j=0}^{k-1}A\left( 
\underline{\eta }_{t-j}\right) \right\} =A^{k}$, then the condition $\left( %
\ref{Eq3_1}\right) $ implies that $\dsum\limits_{k\geq 0}A^{k}<+\infty $ and
hence $\lim\limits_{k\rightarrow \infty }A^{k}=0$. Combining this result
with the fact that $\lim\limits_{k\rightarrow \infty }\left\{
\dprod\limits_{j=0}^{k-1}A\left( \underline{\eta }_{t-j}\right) \right\}
\geq 0$ yields that almost surely $\lim\limits_{k\rightarrow \infty }\left\{
\dprod\limits_{j=0}^{k-1}A\left( \underline{\eta }_{t-j}\right) \right\} =0$.
\end{condition3}

\begin{condition4}
\label{A*.3} Let $\underline{z}\in \mathbb{R}^{sd}$ and the Lyapunov
function $g:\mathbb{R}^{sd}\rightarrow \left[ 1,\infty \right[ $ by $g(%
\underline{z})=\left\Vert \underline{z}\right\Vert ^{r}+1$, we have 
\begin{eqnarray*}
E\left\{ g(\underline{Y}_{t})|\underline{Y}_{t-1}=\underline{y}\right\}
&=&E\left\{ \left\Vert A(\underline{\eta }_{t})\underline{Y}_{t-1}+%
\underline{B}(\underline{\eta }_{t})\right\Vert ^{r}|\underline{Y}_{t-1}=%
\underline{y}\right\} +1 \\
&\leq &E\left\Vert A(\underline{\eta }_{0})\right\Vert ^{r}\left\Vert 
\underline{y}\right\Vert ^{r}+E\left\Vert \underline{B}(\underline{\eta }%
_{0})\right\Vert ^{r}+1 \\
&\leq &\rho \left\Vert \underline{y}\right\Vert ^{r}+\kappa
\end{eqnarray*}%
where $\kappa $ is a positive constant such that $E\left\Vert \underline{B}(%
\underline{\eta }_{0})\right\Vert ^{r}\leq \kappa -1$. Choose $\lambda >0$
with $1-\lambda >\rho $ and set $M=\frac{\kappa }{1-\lambda -\rho }$ and
consider the compact $K=\left\{ \underline{y}\in \mathbb{R}^{sd}:\left\Vert 
\underline{y}\right\Vert \leq M^{\frac{1}{r}}\right\} $. For all $\underline{%
y}\notin K$, we have $\left( 1-\lambda -\rho \right) $ $\left\Vert 
\underline{y}\right\Vert ^{r}\geq \kappa $ and therefore%
\begin{equation}
E\left\{ g(\underline{Y}_{t})|\underline{Y}_{t-1}=\underline{y}\right\} \leq
\rho \left\Vert \underline{y}\right\Vert ^{r}+\left( 1-\lambda -\rho \right)
\left\Vert \underline{y}\right\Vert ^{r}\leq \left( 1-\lambda \right) g(%
\underline{y}).  \tag{6.2}  \label{Eq6_2}
\end{equation}%
When $\underline{y}\in K,$ we have%
\begin{equation}
E\left\{ g(\underline{Y}_{t})|\underline{Y}_{t-1}=\underline{y}\right\} \leq
\rho \left\Vert \underline{y}\right\Vert ^{r}+\kappa \leq \frac{\rho \kappa 
}{1-\lambda -\rho }+\kappa :=b  \tag{6.3}  \label{Eq6_3}
\end{equation}%
Combining $\left( \ref{Eq6_2}\right) $ and $\left( \ref{Eq6_3}\right) $ we
obtain%
\begin{equation*}
E\left\{ g(\underline{Y}_{t})|\underline{Y}_{t-1}=\underline{y}\right\} \leq
\left( 1-\lambda \right) g(\underline{y})+b\mathbf{I}_{K}(\underline{y})
\end{equation*}%
where $\mathbf{I}_{K}$ denotes the indicator function of the $K$.
\end{condition4}
\end{proof}

\section{Conclusion}

The paper partially extends\ $\mathbb{L}^{2}$ structures of the usual $GARCH$
model to periodic ones which allow the volatility of time series to have
different dynamics according to the model parameters which switches between $%
s$-regimes. Our study is based (in a multivariate framework) on a
generalized autoregressive representation which we are preferred for a
periodic $ARMA$ ($PARMA)$ representation. The main advantage of the approach
is that, besides its simplicity, it preserve the mathematically tractable $%
GARCH$ structure when $s=1$. A thorough examination of the\ $\mathbb{L}^{2}$
structures of the series and its powers revealed that, under appropriate
moment conditions, these structures were those of periodic $AR$ ($PAR)$
processes. Beside the conditions ensuring the existence and uniqueness of
strictly stationary solution of $PGARCH$, we have also gave sufficient
conditions for the $PGARCH$ processes to belong to $\mathbb{L}^{p}$, $p\geq
1 $. Some fundamental probabilistic properties such as the $\beta $-mixing
and the geometric ergodicity with exponential decay have been studied.

\end{document}